\documentclass{amsart}
\def\margin_comment#1{\marginpar{\sffamily{\tiny #1\par}\normalfont}}
\usepackage{graphicx}
\usepackage{setspace}
\usepackage[dvips]{color}
\usepackage{amsmath,amssymb,amsthm}
\usepackage{epsfig}

\newtheorem{thm}{Theorem}[section]
\numberwithin{equation}{section} 
\numberwithin{figure}{section} 
\theoremstyle{plain}
\newtheorem*{thm*}{Theorem}
\theoremstyle{definition}
\theoremstyle{plain}

\newtheorem*{defn*}{Definition}
\newtheorem*{ques}{Question}
\theoremstyle{plain}
\newtheorem{prop}[thm]{Proposition} 
\theoremstyle{remark}
\newtheorem{ex}[thm]{Example}
\theoremstyle{remark}
\newtheorem{rem}[thm]{Remark}
\theoremstyle{plain}

\theoremstyle{plain}
\newtheorem{cor}[thm]{Corollary}
\theoremstyle{plain}
\newtheorem{lem}[thm]{Lemma} 
\theoremstyle{definition}
\newtheorem{defn}[thm]{Definition}

\newcommand{\Div}[1]{\operatorname{D}(#1)}
\newcommand{\lDiv}[2]{\operatorname{D}_{\!#1}(#2)}
\newcommand{\CC}{(\mathcal{C})}
\newcommand{\fr}[1]{\Theta_{#1}}
\newcommand{\XX}{{\bf X}}
\newcommand{\SSS}{{\bf S}}
\newcommand{\sXX}{{\bf \scriptscriptstyle X}}
\newcommand{\sX}{{\scriptscriptstyle X}}
\newcommand{\xx}{{\bf x}}
\newcommand{\yy}{{\bf y}}
\newcommand{\zz}{{\bf z}}

\newcommand{\lX}{\ell_{\scriptscriptstyle X}}
\newcommand{\HH}{\phi}
\newcommand{\KK}{\psi}

\AtBeginDocument{
  
}
\begin{document}
\title{Finite quotients of groups of I-type}
\author{Fabienne Chouraqui and Eddy Godelle}
\thanks{Both authors are partially supported by the {\it Agence Nationale de la Recherche} ({\it projet
Th\'eorie de Garside}, ANR-08-BLAN-0269-03). The first author was supported by the Israel Science Foundation ($n^0$ 580/07) and also partially supported by the Affdu-Elsevier fellowship.}
\date{\today}
\begin{abstract}To every group of $I$-type, we associate a finite quotient group that plays the role that Coxeter groups play for Artin-Tits groups. Since groups of I-type are examples of Garside groups, this answers a question of D. Bessis in the particular case of groups of I-type. Groups of $I$-type are related to finite set theoretical solutions of the Yang-Baxter equation.
\end{abstract}
\maketitle
\section*{Introduction}
The motivation that led to develop Garside group theory at the end of the 1990's~\cite{DePa} or, more recently, to  develop Garside family theory~\cite{ddkgm} was to extract the main ideas of Garside's theory of braids~\cite{garside} and to provide a general framework that can be used to understand the algebraic structure of other groups or, more generally, categories. This approach led to many developpments in the last decade and it turns out that most of the main objects that appear in the context of braid groups can be generalized to Garside theory framework. Braid groups are nicely related to symmetric groups. More precisely, the symmetric group on $n$ elements is a quotient of the braid group on $n$ strands. Tits~\cite{tits} extends this result by associating a so-called Artin-Tits group to each Coxeter group, so that the latter is a quotient of the former. Braid groups and, more generally, Artin-Tits groups associated with finite Coxeter groups are seminal examples of Garside groups. Moreover, Coxeter group theory  is a crutial  tool for the study of Artin-Tits groups. Therefore, a natural problem, which was addressed by Bessis in~\cite{bessis2}, is to decide which Garside groups can be associated an object (a \emph{generating generated group}) that plays the role that the symmetric group plays for the braid group (see Section~\ref{sec_genmon} for definitions and a precise question). At the present time, this question remains widely open, even if partial results exist (see~\cite{bessis2}). One attempt to study this question is to consider particular families of Garside groups. In~\cite{chou_art}, the first author has shown that Yang-Baxter theory provides a large family of Garside groups. More precisely, in~\cite{etingof}, Etingof, Soloviev and Schedler associate a group called the \emph{structure group} to each non-degenerate and symmetric set-theoretical solution of the Yang-Baxter equation. It turns out that, firstly, these groups are the so-called groups of $I$-type~\cite{gateva98, jespers}, in other words they are the groups of fractions of monoids that possess a presentation of a particular type (see Theorem~\ref{thm_garsidetableau_structuregp2} below) and, secondly, that these associated monoids  are Garside monoids~\cite{chou_art}. In particular, structure groups are Garside groups. In the present paper we address the question of associating to each structure group, a finite group that plays the role that Coxeter groups play for Artin-Tits groups. One should remark that structure groups are Abelian-by-Finite~\cite[Cor.2.4]{jespers}. So, a naive attempt to answer this question could be to consider the finite quotient group provided by the Abelian-by-Finite structure. However it is easy to verify that this approach does not work (see Section~\ref{secfinquot}). We provide a positive answer for every structure group. Under an extra technical property, denoted by~$\CC$, we obtain a presentation of the finite quotient. Let us postpone some definitions to Section~\ref{sec_bcgd} and state the main result of the paper:
\begin{thm*}(Corollary~\ref{cor_basedon} and Propositions~\ref{lem_cns} and~\ref{label:propfrozensbgp}) Let~$(X,S)$ be a set theoretical solution of the Yang-Baxter. Denote by $n$ the cardinality of $X$, by~$G(X,S)$ its structure group and by~$M(X,S)$ its associated Garside monoid. Then\\
(1)  There is a finite quotient~$W(X,S)$ of $G(X,S)$ that is a generating generated group for $M(X,S)$.\\
(2) If $M(X,S)$ verifies Property~$\CC$, then $W(X,S)$ is a generating generated section for~$M(X,S)$. The order of $W(X,S)$ is $2^n$ and there is an exact sequence \begin{equation} 1\to N(X,S)\to G(X,S)\to W(X,S)\to 1\end{equation} where $N(X,S)$ is a free Abelian group of rank~$n$.
 \end{thm*}

The paper is organised as it follows. In Section~\ref{sec_bcgd}, we introduce the background that we shall need. We recall the notion of a Garside group, the Yang-Baxter equation, the structure group of a set theoretical solution, and the generating group method. In Section~\ref{secfinquot}, we define the group $W(X,S)$ and prove Corollary~\ref{cor_basedon}. In Section~\ref{sec_ftquot}, we focus on the special case where $M(X,S)$ verifies Property~$\CC$. We provide a presentation fo~$W(X,S)$ and prove Propositions~\ref{lem_cns}. and~\ref{label:propfrozensbgp}.
\section{Background}
\label{sec_bcgd}
In this section, we introduce the background that we need. We start with the definitions of a \emph{Garside group} and  of a \emph{Garside monoid}. Then, we introduce those Garside groups that arise as \emph{structure groups}, in other words as groups of $I$-type, and recall how they are related to the Yang-Baxter equation. Finally, we recall the crutial notion for our study, that is the \emph{generating group method}.
\subsection{Garside monoids}
\label{sec_bcgd_gars}
Here, we recall some basic material on Garside theory, and refer to~\cite{deh_francais}, \cite{ddkgm} for more details.
We start with some preliminaries. If $M$ is a monoid generated by a set $X$, and if $g\in M$ is the image of the word~$w$ by the canonical morphism from the free monoid on $X$ onto $M$, then we say that \emph{$w$ represents $g$} or, equivalently, that \emph{$w$ is a word repesentative of $g$}. A monoid~$M$ is \emph{cancellative} if for every~${e,f,g,h}$ in $M$, the equality~$efg = ehg$ implies~$f = h$. The element $f$ is a \emph{left divisor}  ({\it resp.} a \emph{right divisor}) of $g$ if there is an element $h$ in $M$ such that $g=fh$ ({\it resp.} $g = hf$).  It is left noetherian ({\it resp.}  \emph{right noetherian}) if every sequence~$(g_n)_{n\in\mathbb{N}}$ of elements of $M$ such that $g_{n+1}$ is a left divisor ({\it resp.} a right divisor) of $g_n$ stabilizes. It is noetherian if it is both left and right noetherian. An element~$\Delta$ is said to be \emph{balanced} if it has the same set of right and left divisors. In this case, we denote by~$\Div{\Delta}$ its set of divisors. If~$M$ is a cancellative and noetherian monoid, then left and right divisibilities are partial orders on~$M$.

\begin{defn} (1) A \emph{locally Garside monoid} is a cancellative noetherian monoid such that
\begin{enumerate}
 \item[(a)] any two elements have a common multiple for left-divisibility  if and only if they have a least common multiple for left-divisibility;
\item[(b)] any two elements have a common multiple for  right-divisibility if and only if they have a least common multiple for right-divisibility.
\end{enumerate}
(2)  A \emph{Garside element} of a locally Garside monoid is a balanced element~$\Delta$ whose set of divisors~$\Div{\Delta}$ generates the whole monoid. In this case, $\Div{\Delta}$ is called a \emph{ Garside family} of $M$.\\
(3) A monoid is a \emph{Garside monoid} if it is a locally Garside monoid with a Garside element whose set of divisors~$\Div{\Delta}$ is finite.\\
(4) A (\emph{locally}) \emph{Garside group}~$G(M)$ is the enveloping group of a (locally) Garside monoid~$M$.
\end{defn}
Garside groups have been first introduced in~\cite{DePa}. The seminal examples are the \emph{spherical type Artin-Tits groups}. We refer to~\cite{DiM} for general results on locally Garside groups. Recall that an element $g\neq 1$ in a monoid is called an \emph{atom} if the equality $g = fh$ implies $f = 1$ or $h = 1$. It follows from the definining properties of a Garside monoid that the following properties hold for a Garside monoid~$M$: The monoid~$M$ is generated by its set of atoms, and every atom divides the Garside elements. there is no invertible element, except the trivial one, and any two elements in $M$ have a left ({\it resp. right}) gcd and a left ({\it resp. right}) lcm; in particular, $M$ verifies the Ore's conditions, so it embeds in its group of fractions~\cite{Clifford}; in the sequel we will always consider $M$ as a submonoid of its group of fractions. The left and right gcd of two Garside elements are Garside elements and coincide; therefore, by the noetherianity property there exists a unique minimal Garside element for both left and right  divisibilities. This element~$\Delta$ will be called \emph{the} Garside element of the monoid and the set $\Div{\Delta}$ will be called \emph{the} Garside family of~$M$, and the elements of~$\Div{\Delta}$ will be called the \emph{simple} elements of $M$. Finally it is important to notice that if $\Delta$ is a balanced element then, $\Div{\Delta}$ is \emph{closed under factors}: if $fgh$ belongs to $\Div{\Delta}$, then $f,g$ and $h$ belong to $\Div{\Delta}$.

\subsection{Set theoretical solution of the Quantum Yang-Baxter Equation}
\label{secQYBE}
Here, we introduce basic notions related to the Quantum Yang-Baxter Equation and the main objects of our study, that is, structure groups.  We follow~\cite{etingof} and refer to it for more details.

Fix a finite dimensional vector space~$V$ on the field~$\mathbb{R}$. The Quantum Yang-Baxter Equation on~$V$ is the equality \begin{equation}R^{12}R^{13}R^{23}=R^{23}R^{13}R^{12}\end{equation} of linear transformations on $V  \otimes V \otimes V$ where the indeterminate is a linear transformation~$R: V  \otimes V\to V  \otimes V$, and $R^{ij}$ means $R$ acting on the $i$th and $j$th components. A \emph{set-theoretical solution} of this equation is a pair~$(X,S)$ such that $X$ is a basis for $V$, and $S : X \times X \rightarrow X \times X$ is a bijective map that induces a solution~$R$ of the QYBE.
Following \cite{etingof}, we introduce the convenient functions~$g_x:X\to X$ and $f_x:X\to X$ for $x$ in $X$ by setting  \begin{equation}S(x,y)=(g_{x}(y),f_{y}(x)).\end{equation} The pair~$(X,S)$ is said to be \emph{nondegenerate} if for any  $x\in X$, the maps $f_{x}$ and $g_{x}$ are bijections. It is said to be \emph{symmetric} if  it is \emph{involutive}, that is $S\circ S = Id_X$, and  \emph{braided}, that is $S^{12}S^{23}S^{12}=S^{23}S^{12}S^{23}$, where the map $S^{ii+1}$ means $S$ acting on the $i$th and $(i+1)$th components of $X^3$.

\begin{defn} \label{def_struct_gp}Assume~$(X,S)$ is non-degenerate and symmetric. The \emph{structure group} of $(X,S)$ is defined to be the group~$G(X,S)$ with the following group presentation:
\begin{equation}\langle X\mid\ xy = g_x(y)f_y(x)\ ;\ x,y\in X,\ S(x,y)\neq (x,y) \rangle\label{equation:structuregroup1}.\end{equation}
\end{defn}

Since the maps $g_x$ are bijective and $S$ is involutive, one can deduce that for each $x$ in $X$ there are unique $y$ and $z$ such that~$S(x,y) = (x,y)$ and~$S(z,x) = (z,x)$. Therefore, the presentation of~$G(X,S)$ contains $\frac{n(n-1)}{2}$ non-trivial relations. In the sequel, we denote by $M(X,S)$ the monoid defined by the monoid presentation~(\ref{equation:structuregroup1}). In particular, $G(X,S)$ is the enveloping group of $M(X,S)$.

Let $\alpha:X \times X \rightarrow X\times X$ be defined by $\alpha(x,y)=(y,x)$, and let $R=\alpha \circ S$. The map $R$ is the so-called \emph{$R$-matrix} corresponding to $S$. Etingof, Soloviev and Schedler show in \cite{etingof} that $(X,S)$ is a braided pair if and only
 if $R$ satisfies the QYBE. A solution~$(X,S)$ is said to be \emph{trivial} if the maps $f_{x}$ and $g_{x}$ are the identity on $X$ for all  $x\in X$, that is if $S$ is the map~$\alpha$ defined above.

The connection between set theoretical solutions of the Yang-Baxter equation and Garside groups has been established by the first author. Before stating it, let us recall the following definition:
\begin{defn}\label{def:Itype} A monoid $M$ is a monoid of $I$-type if it admits a finite monoid presentation~$\langle X\mid R \rangle$ such that:
\begin{enumerate}
\item[(a)] the cardinality of $R$ is $n(n-1)/2$, where $n$ is the cardinality of $X$, and each relation in $R$ is of the type $xy = zt$ with $x,y,z,t\in X$;
\item[(b)] every word $xy$, with $x,y$ in $X$, appears at most once in $R$.
\end{enumerate}
\end{defn}
Actually this is not the initial definition of a monoid of $I$-type but the one given here has been shown to be equivalent to the initial one. We will say that a group is of $I$-type if and only if it is the envelopping group of a monoid of $I$-type. The above presentation will be called a  presentation of $I$-type.

\begin{thm}\label{thm_garsidetableau_structuregp2} (1) \cite{gateva98, jespers} Let $G$ be a group. Then, the group~$G$ is of $I$-type if and only if~$G$ is a structure group. More precisely,
\begin{enumerate} \item[(a)] if~$(X,S)$ is a non-degenerate symmetric set-theoretical solution~$(X,S)$, then $M(X,S)$ is a monoid of $I$-type and the presentation~(\ref{equation:structuregroup1}) is a presentation of $I$-type;\item[(b)] assume~$M$ is a  Garside monoid that admits a presentation~$\langle X\mid R \rangle$ of $I$-type, then there exists a map~$S: X \times X \rightarrow X \times X$ such that $(X,S)$ is  a non-degenerate symmetric set-theoretical solution. Moreover, the presentation in~(\ref{equation:structuregroup1}) is~$\langle X\mid R \rangle$. In particular~$M(X,S) = M$. \end{enumerate}
(2) \cite{chou_art} For every non-degenerate symmetric set-theoretical solution~$(X,S)$, the structure group $G(X,S)$ is a Garside group, whose Garside monoid is~$M(X,S)$. Moreover, the atom set of the monoid~$M$ is $X$, and the Garside element is both the left lcm and the right lcm of~$X$.
\end{thm}

As explained in Section~\ref{sec_bcgd_gars}, in the sequel $M(X,S)$  (can and) will be identified with the submonoid of~$G(X,S)$ generated by~$X$.

\begin{ex} \label{exemple:exesolu_et_gars} Set $X = \{x_1,x_2,x_3,x_4\}$, and let~$S: X\times X\to X\times X$ defined by $S(x_i,x_j)=(x_{g_{i}(j)},x_{f_{j}(i)})$ where~$g_i$ and $f_j$ are permutations on~$\{1,2,3,4\}$ as follows: $g_{1}=(2,3)$, $g_2=(1,4)$, $g_{3}=(1,2,4,3)$, $g_{4}=(1,3,4,2)$; $f_{1}=(2,4)$, $f_2=(1,3)$, $f_{3}=(1,4,3,2)$, $f_{4}=(1,2,3,4)$.
A direct analysis shows that $(X,S)$ is a non-degenerate symmetric set theoretical solution. The defining presentation of $G(X,S)$ contains six non trivial relations $$\begin{array}{ccccc}
x_{1}x_{2}=x^{2}_{3};& x_{1}x_{3}=x_{2}x_{4};&
x_{2}x_{1}=x^{2}_{4};\\ x_{2}x_{3}=x_{3}x_{1};&
x_{1}x_{4}=x_{4}x_{2};&x_{3}x_{2}=x_{4}x_{1}\\
\end{array}$$\\
and four trivial relations.
\end{ex}

\subsection{The generated group method}
\label{sec_genmon}
We turn now to the notion of a \emph{generating finite group of a monoid}. We almost follow~\cite{bessis2} (see also \cite{michel}). Let $W$ be a group equiped with a set~$X$ that generates~$W$ as a monoid. The pair~$(W,X)$ will be called a \emph{generated group}. We define the length~$\ell_X(w)$ of an element~$w$ in $W$ as the minimal length of a word on~$X$ that represents~$w$.  A \emph{reduced expression} of an element $w$ in $W$ is a word representative $x_1x_2\cdots x_k$ on $X$ such that $\lX(w)=k$. When $w_1,w_2,w_3$ belong to~$W$ such that $w_1 = w_2w_3$ with  $\ell(w_1) = \ell(w_2)+\ell(w_3)$, we say that $w_2$ and $w_3$ are a \emph{left $X$-factor} and a \emph{right $X$-factor} of $w_1$, respectively. Because of the condition on the length in the definition of a left $X$-factor, the  relation ``$w$ is a left $X$-factor of $h$'' is a partial order on $W$. Similarly,  the notion of a right $X$-factor induces a partial order on $W$, too. We say that an element~$w$ in~$W$ is $X$-\emph{balanced} if its sets of left $X$-factors and right $X$-factors coincide. In this case, we denote this set by $\lDiv{\sX}{w}$. Now, for every $X$-balanced element~$w$ in $W$, and given a copy  $\{\underline{v}, v\in \lDiv{\sX}{w}\}$ of $\lDiv{\sX}{w}$, we define a monoid~$M_{w,X}$ by the following monoid presentation:
\begin{equation}\left\langle  \underline{v}, v\in \lDiv{\sX}{w}\mid \underline{v}\,\underline{v'} = \underline{v''}\textrm{ when } v,v',v''\in \lDiv{ \sX}{w},\ \left\{\begin{array}{l} vv' = v''\\\ell_\sX(v)+\ell_\sX(v') = \ell_\sX(v'')\end{array}\right.\right\rangle.\end{equation}

\begin{ex} (1) Take $W = \langle s\mid s^2 = 1\rangle$, $X = \{s\}$ and $w = s$ then~$M_{s,\{s\}}$ is $\{\underline{s}^j\mid j\in\mathbb{N}\}$. \\
(2) Take $W = \langle s\mid s^4 = 1\rangle$, $X = \{s,s^{-1}\}$ and $w = s^2$ then~$M_{s^2,\{s,s^{-1}\}}$ is $\langle a,b\mid a^2 = b^2\rangle$ with $a = \underline{s}$ and $b = \underline{s^{-1}}$.
\end{ex}

Now, one should note that for a finite group~$W$, a subset $X$ of $W$ generates~$W$ as a group if and only if it generates $W$ as a monoid.

\begin{defn} We say that~$(W,X)$ is a \emph{generating generated group} for a monoid $M$ if $W$ is a finite group that contains a $X$-balanced element~$w$ such that $\lDiv{\sX}{w}$ generates $W$ and $M$ is isomorphic to~$M_{w,X}$. When furthermore, $\lDiv{\sX}{w} = W$, we say that~$(W,S)$ is a \emph{generating generated section} for $M$.
\end{defn}

As long as it will not introduce confusion, we will often say that $W$ is a generating generated group ({\it resp.} a generating generated section) for $M$ instead of $(W,X)$ is a generating generated group ({\it resp.} a generating generated section) for $M$. It is easy to see that the map $W\to M_{w,X},\ v\mapsto \underline{v}$ is into and there is a morphism of monoids~$p: M_{w,X}\to W$ defined by $p(\underline{v}) =  v$. Also, the length function~$\ell_{\underline{\sX}} : M_{w,X}\to \mathbb{N}$ is additive and for every~$a\in M_{w,X}$, one has $\lX(p(a))\leq \ell_{\scriptscriptstyle \underline{X}}(a)$.  As a consequence, $\{\underline{v}\mid v\in X\cap\lDiv{\sX}{w}\}$ is the atom set of $M_{w,X}$. Moreover, in the special case where $\lDiv{\sX}{w} = W$, we have $\lX(p(a)) =  \ell_{\underline{\sX}}(a)$  if and only if $a$ belongs to~$\Div{\underline{w}}$.
Here, the crucial result is
\begin{thm} \label{the_jean}\cite{bessis2,michel} Let $W$ be a finite group and $X$ be a generating set. Assume~$w$ is $X$-balanced in~$W$ and that $\lDiv{\sX}{w}$ is a lattice for both partial orders associated with the left and right $X$-factor notions. Then
$M_{w,X}$ is a Garside monoid with~$\{\underline{s}\mid s\in X\cap\lDiv{\sX}{w}\}$ as atom set. The element~$\underline{w}$ is a Garside element of $M_{w,X}$ with~$\Div{\underline{w}} = \{\underline{v}\mid v\in\lDiv{\sX}{w}\}$.
 \end{thm}
Any spherical type Artin-Tits monoid~$A^+$ has a generating generated section~$(W,X)$: in this case the group $W$ is the associated Coxeter group equipped with its standard generating set~$X$; the element~$w$ is the Coxeter element~$w_0$ of $W$. It has been shown by Bessis in~\cite{bessis2} that dual braid monoids have a generating generated group. In this case, the group~$W$ is again the associated Coxeter group, the set $X$ is the set of all reflections, and $w$ is a Coxeter element. It could be noted that in the case of dual braid monoids, $\lDiv{\sX}{w}$ is not the whole group. These results led Bessis~\cite{bessis2} to address the following question: which Garside monoids have a generating generated group? Clearly, one
can not expect that every Garside monoid has a generating generated group, since there are Garside monoids with non-homogenous length function associated with their atom set. So, the question is restricted to those Garside monoids that possess an additive length function on their atom set. This is clear that monoids of $I$-type are Garside monoids that satisfy this restriction, as the defining relations are homogenous (see Section~\ref{secQYBE}).
\section{A linear representation of $G(X,S)$}
\label{secfinquot}
{Let $X$ be a finite set of cardinality $n$, and $(X,S)$ be a non-degenerate symmetric set-theoretical solution of the QYBE, defined by $S(x,y) = (g_x(y),f_y(x))$, where~$g_x:X\to X$ and~$f_x:X\to X$ are bijective. Let~$M(X,S)$ and~$G(X,S)$ be respectively the corresponding Garside  monoid and Garside group. We denote by~$\Delta$ the Garside element of~$M(X,S)$. We recall that $X$ belongs to~$\Div{\Delta}$ and is the atom set of~$M(X,S)$. In this section, we define a linear representation of~$G(X,S)$ that permits us to answer Bessis's question in the positive. In other words, we associate a finite generating generated group~$W(X,S)$ to~$G(X,S)$, and later on (in Section 3) we find a necessary and sufficient condition on $(X,S)$ so that~$W$ is a generating generated section.

The group~$G(X,S)$ is a group of $I$-type. Jespers and Okninski showed that groups of~$I$-type are Abelian-by-Finite (see \cite{jespers}). Indeed, they showed that if $G(X,S)$ is a group of $I$-type associated with a set theoretical solution~$(X,S)$, where $X$ has cardinality $n$, then
$G(X,S)$ is a subgroup of the (obvious) semi-direct product~$FA_n\rtimes \textrm{Sym}_n$, where
    $FA_n$ is the free Abelian group on $n$ generators and $\textrm{Sym}_n$ is the symmetric group on~$\{1,\cdots,n\}$. Moreover, the first projection~$G(X,S)\to FA_n$ is one-to-one, and there is a subgroup~$W$ of $\textrm{Sym}_n$ and an Abelian subgroup~$A$ of~$FA_n$, such that the sequence \begin{equation}1\to A\to G(X,S)\to W\to 1\end{equation} is exact. \\
     A question that arises naturally is whether this exact sequence can provide a generating generated group for~$G(X,S)$, or in other words whether this group~$W$ is a generating generated group for~$G(X,S)$. As the following example illustrates it, the answer is negative. Take~$G(X,S)$ to be the free Abelian monoid~$FA(x,y)$ on $x,y$. It is a group of $I$-type with the presentation $\langle x,y\mid xy = yx\rangle$. If this approach worked, as $FA(x,y)$ is a Artin-Tits group of spherical type, $W$ should be the Coxeter group~$\langle x,y\mid x^2 = y^2 = 1 ; xy = yx\rangle$. But, unfortunately, it is easy to see that the group~$W$ provided by the above exact sequence is the trivial group, so this approach does not work, and $W$ is not be a generating generated group for~$G(X,S)$ in general.

\subsection{Frozen elements and simple elements}
The main object of this section is to recall a technical result, namely Proposition~\ref{prop_simple_lcm1}, which turns out to be a crucial argument in the sequel. We recall that for every  non-degenerate symmetric solution~$(X,S)$ and every~$x$ in~$X$ there exists a unique $y$ in $X$ such that~$S(x,y) = (x,y)$. In the sequel, we call such a pair~$(x,y)$ a \emph{frozen} pair. In this case, the word ({\it resp.} the element)~$xy$ will be called a frozen word ({\it resp.} a frozen element). A frozen element has therefore a unique word representative, the associated frozen word.

The main result here is that a simple element cannot be represented by a word containing a frozen word as a subword. As already remarked, the defining relations in the presentation~(\ref{equation:structuregroup1}) are homogenous, so we can define a length function~$\ell: M(X,S)\to \mathbb{N}$ so that the length of an element is the length of any of its word representatives on~$X$.
\begin{prop}
\cite{chou_art,chou_godel}\label{prop_simple_lcm1}
Let $a$ be in $M(X,S)$. Denote by $X_l(a)$ the set of its left divisors that belongs to~$X$ and by $X_r(a)$ the set of its right divisors that belong to~$X$. Then, $a$  belongs to~$\Div{\Delta}$ if and only if it is the right lcm of of $X_l(a)$  if and only if it is the left lcm of $X_r(a)$. Moreover in this case, $X_l(a)$ and $X_r(a)$ have the same cardinality, which is $\ell(a)$.
\end{prop}
From the above result, we deduce that
\begin{prop}\label{prop_simple_lcm2}
 Let $a$ be an element of $M(X,S)$ and $x$ be in $X$.\\
(1) Assume  $a$ belongs to~$\Div{\Delta}$ and  $xa$ does not belong to~$\Div{\Delta}$.  Then there exist $y$ in $X$ and $b$ in $\Div{\Delta}$ such that $a = yb$ and $(x,y)$ is a frozen pair.\\
(2) $a$ belongs to~$\Div{\Delta}$ $\Leftrightarrow$ no expression of $a$ contains a frozen word as a subword.\\
(3) Assume $y$ lies in $X$ so that both $xa$ and $ay$ are in $\Div{\Delta}$ but $xay$ is not. If $a = z_1\cdots z_k$ with $z_1,\ldots,z_k$ in $X$, then there exist $y_1,\ldots,y_{k+1}$ in $X$ and $x_1,\ldots,x_k$ in $X$ so that $y_{k+1} = y$, $S(z_i,y_{i+1}) = (y_i,x_i)$, no $(z_i,y_{i+1})$ is a frozen pair and $(x,y_1)$ is a frozen pair. In particular, $ay = z_1\cdots z_{i-1}y_ix_i\cdots  x_k = y_1x_1\cdots x_k$ in $M(X,S)$.
\end{prop}
\begin{proof}
$(1)$ Assume $a$ belongs to~$\Div{\Delta}$ and set $k = \ell(a)$. By Proposition~\ref{prop_simple_lcm1}, there exist $k$ distinct elements~$x_1,\ldots,x_k$ in~$X$ and $k$ elements~$a_1,\ldots, a_k$ of~$M(X,S)$ such that $a = x_ia_i$ for $i = 1,\ldots,k$. Now, assume that no pair~$(x,x_i)$ is a frozen pair. Then we have defining relations $xx_i = y_iz_i$. Moreover $y_i = g_x(x_i)$ and $g_x$ is a bijection so that $g_x(z) = x$ if $(x,z)$ is a frozen pair. Then, all the $y_i$ are distinct and distinct from $x$. It follows that $xa$ has to be left divisible by the lcm of $x,y_1,\cdots, y_k$. But, by Proposition~\ref{prop_simple_lcm1}, the length of this lcm is~$k+1$, that is $\ell(xa)$. Thus $xa$ has to belong to~$\Div{\Delta}$. Hence, since from assumption, $xa$ does not belong to~$\Div{\Delta}$ there is some $i$ in $\{1,\ldots,k\}$ so that $(x,y_i)$ is a frozen pair. Finally, $a$ belongs to $\Div{\Delta}$ and $a_i$ is a factor of~$a$, therefore, $a_i$ belongs to~$\Div{\Delta}$.\\
$(2)$ It follows from Proposition~\ref{prop_simple_lcm1} that frozen words does not belong to~$\Div{\Delta}$: they have length two and only one left divisor. As $\Div{\Delta}$ is closed by factors, no expression of an element in~$\Div{\Delta}$ can contain a frozen word. Conversely, assume $a$ is not in $\Div{\Delta}$ and write $a = x_1\cdots x_k$ with $x_1,\ldots,x_k$ in $X$. As $x_k$ is in $\Div{\Delta}$ and $a$ is not, there is a subscript~$i$ such that $x_{i+1}\cdots x_k$ is in $\Div{\Delta}$ whereas $x_i\cdots x_k$ is not. By (1), there exist $y_{i+1},\ldots,y_k$ in $X$ so that $y_{i+1}\cdots y_{k} = x_{i+1}\cdots x_k$ and $(x_i,y_{i+1})$ is a frozen word. But $a = x_1\cdots x_iy_{i+1}\cdots y_{k}$. Hence, there is an expression of $a$ that contains a frozen word.\\
 (3) We prove the result by induction on $k$. If $k = 0$, there is nothing to prove. Assume $k\geq 1$. Let $y_1$ be in $X$ so that $(x,y_1)$ is a frozen pair. Since $xa$ belongs to~$\Div{\Delta}$, it follows from (2) that~$y_1$ and $z_1$ have to be distinct, and from~(1) that both left divide~$ay$. As the map $g_{z_1}: X\to X$ is a bijection, and $M(X,S)$ is cancellative, there exists a unique pair $(y_2, x_1)$ of element of $X$, so that $S(z_1,y_2) = (y_1,x_1)$. This imposes that $(z_1,y_2)$ is not a frozen pair,  $z_1y_2$ and $y_1x_1$ are equal in $M(X,S)$, and are the right lcm of $y_1$ and $z_1$ by Proposition~\ref{prop_simple_lcm1}. Hence, by cancellativity, $y_2$ left divides~$z_2\cdots z_k y$. Denote by~$x'_2$ the unique element of $X$ so that $(x'_2,y_2)$ is a frozen pair. Then, $x'_2z_2\cdots z_k y$ is not in $\Div{\Delta}$ by~(2) and $z_2\cdots z_k y$ is in $\Div{\Delta}$  since it right divides $ay$. Now, we claim that $x'_2z_2\cdots z_k$ lies in $\Div{\Delta}$ too. Otherwise, by (1), we could write $z_2\cdots z_k = y_2a'$ and we would have $xa = xz_1y_2a' = xy_1x_1a'$, which is impossible as $xa$ belongs to~$\Div{\Delta}$. So, we are in position to apply the induction hypothesis: there exist $y_3,\ldots,y_{k+1}$ in $X$ and $x_2,\ldots,x_k$ in $X$ so that $y_{k+1} = y$, $S(z_i,y_{i+1}) = (y_i,x_i)$, and no $(z_i,y_{i+1})$ is a frozen pair. \end{proof}

\begin{prop}\label{prop_nombre_simple}
  (1) If $a$ is in~$\Div{\Delta}$ then it has $\ell(a)!$ representative words.\\
  (2) The number of simple elements of length $k$ is $\frac{n!}{(n-k)!k!}$ and so the cardinality of~$\Div{\Delta}$ is~$2^n$.   \end{prop}
\begin{proof}
$(1)$ This is immediate by induction on $\ell(a)$. If $\ell(a) = 1$, then $a$ is an atom and has a unique representative word : no relation can be applied to this representative word. Now, assume~$\ell(a)\geq 2$. Let $x$ be in $X$ that is a left divisor of $a$, and write $a = xa_1$ with $a_1$ in~$M(X,S)$. Then $a_1$ belongs to~$\Div{\Delta}$ and has length~$\ell(a)-1$. By the induction hypothesis, $a_1$ has $(\ell(a)-1)!$ representative words. Thus, there is $(\ell(a)-1)!$ word representatives of $a$ that start with~$x$. But by Proposition~\ref{prop_simple_lcm1}, $a$ has $\ell(a)$ distinct left divisors in~$X$ (each of them has a unique representative word). Then  $a$ has $\ell(a)\times (\ell(a)-1)!$, that is $\ell(a)!$  representative words.\\
$(2)$ It follows from Proposition \ref{prop_simple_lcm1}$(1)$ that there is a bijection between the subsets of $X$ whose  cardinality is equal to~$k$ and the elements of~$\Div{\Delta}$ whose length is $k$. The result follows.
\end{proof}

\subsection{A representation of~$G(X,S)$}
We are now ready to define the representation~$\KK$ of~$G(X,S)$. We are going to define a first representation~$\HH$, that will not be suitable, and then modify it to obtain the expected representation~$\KK$. In the sequel we denote by $\HH_x:X\to X$ the map $f_x^{-1}$. So, for every $x,y$ in $X$, one has $\HH_x(y)x = \HH_y(x)y$.

So, be definition, $y$ right divides $\HH_x(y)x$ in $M(X,S)$ and, for $y\neq x$, the element $\HH_x(y)x$ is the left lcm of $x$ and $y$ in $M(X,S)$. Moreover, for any $x$ in $X$, the pair $(\HH_x(x),x)$ is frozen. It will be convenient to use the diagramatic representation of Figure~\ref{diagramHmap}. Now, we extend the map $\HH: x\to \HH_x$ to a morphism~$\HH : w\to \HH_w$ of the free group on $X$  to the symmetric group~$\mathfrak{S}(X)$ on $X$. So, for a word~$w = x^{\varepsilon_1}_1 x^{\varepsilon_2}_2\cdots x^{\varepsilon_k}_k$ with $x_1,\ldots,x_k$ in $X$ and~$\varepsilon_1,\ldots,\varepsilon_k$ in $\{\pm 1\}$, we have $\HH_w = \HH^{\varepsilon_1}_{x_1}\circ \HH^{\varepsilon_2}_{x_2}\circ\cdots \HH^{\varepsilon_k}_{x_k}$. In the sequel the following easy result will be useful.

\begin{lem}\label{lem_fct_H} Let $y \in X$ and $w \in~X^*$. Then, in $M(X,S)$, the element~$y$ right divides the element represented by the word $\HH_w(y)w$. Moreover, if $y$ does not right-divide $w$, then $\HH_w(y)w$ is the left lcm of $y$ and the element represented by $w$.
\end{lem}
\begin{proof} If $w$ belongs to $X$, then the result holds as remarked above. Now, it is easy to prove the result by induction on the length of the word~$w$ (see Figure~\ref{diagramHmap}), and we left the details to the reader.
\end{proof}

\begin{figure}[ht]
\begin{picture}(270,50)
\put(18,8){\includegraphics[scale = 0.5]{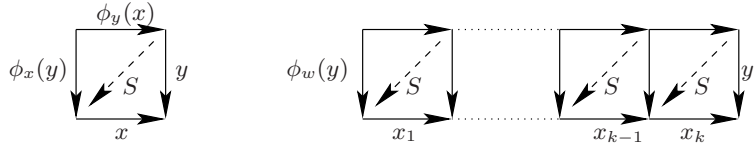}}
\put(35,3){{\small $x$}} \put(28,48){{\small$\HH_y(x)$}}\put(-5,27){{\small$\HH_x(y)$}}\put(58,27){{\small $y$}}\put(38,20){{\small$S$}}

\put(140,3){{\small $x_1$}}\put(216,3){{\small $x_{k-1}$}}\put(249,3){{\small $x_k$}} \put(100,27){{\small$\HH_w(y)$}}\put(272,27){{\small $y$}}
\put(145,20){{\small$S$}} \put(220,20){{\small$S$}}\put(254,20){{\small$S$}}
\end{picture}
\caption{the $\HH$-maps of $x\in X$ and of $w = x_1\cdots x_k$.}
\label{diagramHmap}
\end{figure}
\begin{prop} \label{propHmap} The morphism $\HH:w\to \HH_w$ induces a morphism from~$G(X,S)$ to $\mathfrak{S}(X)$, the symmetric group on $X$.
\end{prop}
It was proved in~\cite{etingof} that the map~$x\mapsto f_x$ defines a right action of~$G(X,S)$ on the set~$X$. The above result easily follows. For completeness, we provide a direct proof.
\begin{proof} Assume $S(x_1,x_2) = (x'_1,x'_2)$ with $x_1,x_2$ in $X$ and $x_1\neq x'_1$. We need to prove that~$\HH_{x_1x_2} = \HH_{x'_1x'_2}$. If $y\neq x_2$ and $y\neq x'_2$, then $y$ cannot right divide $x_1x_2$ by Proposition~\ref{prop_simple_lcm1}. From lemma \ref{lem_fct_H}, this implies that $\HH_{x_1x_2}(y)x_1x_2$ and $\HH_{x'_1x'_2}(y)x'_1x'_2$  are  both equal to the left lcm of $y, x_2$ and $x'_2$. Therefore,  we have $\HH_{x_1x_2}(y)x_1x_2 = \HH_{x'_1x'_2}(y)x'_1x'_2$, and by cancellativity, $\HH_{x_1x_2}(y)$ and $\HH_{x'_1x'_2}(y)$. Assume now $y = x_2$ or $y = x'_2$. As $S$ is involutive, we may assume $y = x_2$. Then $\HH_{x'_1x'_2}(x_2) = \HH_{x'_1}(\HH_{x'_2}(x_2)) = \HH_{x'_1}(x'_1) = y'_1$ so that~$(y'_1,x'_1)$ is a frozen pair (see Figure~\ref{diagramHmap2}). Let $y'_2$ be in $X$ so that $(y'_2,x_2)$ is a frozen pair. Then, we have $\HH_{x_1x_2}(x_2) = \HH_{x_1}(y'_2)$ (see Figure~\ref{diagramHmap2}). Since $(x_1,x_2)$ is not a frozen pair, we have $x_1\neq y'_2$. Using that $\HH_{x_1}(y'_2)x_1 = \HH_{y'_2}(x_1)y'_2$, we get $\HH_{x_1}(y'_2) \neq \HH_{y'_2}(x_1)$. Now, by Proposition~\ref{prop_simple_lcm1}, the element $\HH_{x_1}(y'_2)x_1x_2$ of $M(X,S)$,  that is equal to $\HH_{y'_2}(x_1)y'_2x_2$, is not in~$\Div{\Delta}$ because one of its representing words contains a frozen word, namely $y'_2x_2$. Since its length is $3$, still by Proposition~\ref{prop_simple_lcm1}, its left divisors in $M(X,S)$ that belong to $X$ are $\HH_{x_1}(y'_2)$ and $\HH_{y'_2}(x_1)$ only. But, we have $\HH_{x_1}(y'_2)x_1x_2 = \HH_{x_1}(y'_2)x'_1x'_2 = y''_2x''_1x'_2$ with $(y''_2,x''_1) = S(\HH_{x_1}(y'_2),x'_1)$. Thus $y''_2$ is either equal to $\HH_{x_1}(y'_2)$ or to $\HH_{y'_2}(x_1)$.  Assume the second case holds. Then, from the equality~$y''_2x''_1 = \HH_{x_1}(y'_2)x'_1$, it follows that $\HH_{y'_2}(x_1)x''_1 = \HH_{x_1}(y'_2)x'_1$. But, we have also~$\HH_{y'_2}(x_1)y'_2 = \HH_{x_1}(y'_2)x_1$, and $\HH_{x_1}(y'_2)x_1$ is the right lcm of $\HH_{y'_2}(x_1)$ and $\HH_{x_1}(y'_2)$.   we conclude that  $x_1 =  x'_1$, a contradiction. Therefore, $y''_2 = \HH_{x_1}(y'_2)$. Since we have $(y''_2,x''_1) = S(\HH_{x_1}(y'_2),x'_1)$, this means that $(\HH_{x_1}(y'_2),x'_1)$ is a frozen pair. Using that $(y'_1,x'_1)$ is a frozen pair, we obtain~$\HH_{x_1}(y'_2) = y'_1$, that is $\HH_{x_1x_2}(x_2) = \HH_{x'_1x'_2}(x_2)$.
\end{proof}
\begin{figure}[ht]
\begin{picture}(270,50)
\put(18,8){\includegraphics[scale = 0.5]{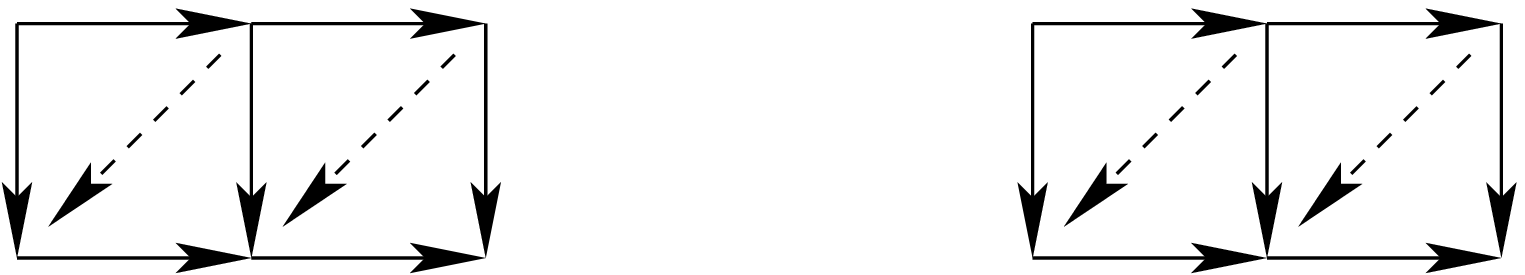}}
\put(35,2){{\small $x'_1$}}\put(68,2){{\small $x'_2$}}\put(38,20){{\small$S$}}\put(71,20){{\small$S$}} \put(35,48){{\small$y'_1$}}\put(11,27){{\small$y'_1$}}
\put(68,48){{\small$x_1$}}\put(56,27){{\small $x'_1$}} \put(91,27){{\small $x_2$}}

\put(180,3){{\small $x_1$}}\put(213,3){{\small $x_2$}}\put(183,20){{\small$S$}}\put(216,20){{\small$S$}} \put(169,48){{\small $\HH_{y'_2}\!(x_1)$}}\put(135,27){{\small $\HH_{x_1}\!(y'_2)$}}
\put(213,48){{\small$y'_2$}}\put(203,27){{\small $y'_2$}} \put(237,27){{\small $x_2$}}
\end{picture}
\caption{$\HH_{x_1x_2}(x_2)$ and $\HH_{x'_1x'_2}(x_2)$.}
\label{diagramHmap2}
\end{figure}

In the sequel, we still denote by~$\HH: G(X,S)\to GL(V),\ g\mapsto \HH_g$ the morphism induced by $\HH:w\mapsto \HH_w$.  As $\mathfrak{S}(X)$ is a finite group, regarding the main question we address in this article, one may wonder whether the image of $G(X,S)$ in $\mathfrak{S}(X)$ by the morphism~$\HH$ is a generating finite group for $M(X,S)$. However, it is easy to see that the restriction to~$\Div{\Delta}$ of the map $\HH$ is not into in general. Consider for instance $M(X,S) = \langle x,y\mid xy = yx\rangle$. Then $\HH_x = \HH_y = Id_X$.  On the other hand, recall from the introduction that $X$ is a base of a finite dimensional vector space~$V$. So, our strategy will be to see $\HH_w$ as an element of $GL(V)$ and to modify the map~$\HH$ in order to obtain a generating group as a finite subgroup of $GL(V)$.

 For $x$ on $X$, we define $\KK_x: V\to V$ to be the linear map defined on the base~$X$ of $V$ by $$\begin{array}{ll}\KK_x(y) = \HH_x(y)&\textrm{ for }y\neq x;\\\KK_x(x) = -\HH_x(x).\end{array}$$  Note that for every $x$ in $X$, the linear map~$\KK_x$ belongs to $GL(V)$, as its matrix in the base~$X$ is monomial (it has a unique non zero entry on each line and on each colomn), with non-zero entries equal to $\pm 1$. Its determinant is therefore~$\pm 1$. As for $\HH$,  we extend the map $\KK: x\to \KK_x$ to a morphism~$\KK : w\to \KK_w$ of the free group on $X$  to~$GL(V)$. So, for a word~$w = x^{\varepsilon_1}_1 x^{\varepsilon_2}_2\cdots x^{\varepsilon_k}_k$ with $x_1,\ldots,x_k$ in $X$ and~$\varepsilon_1,\ldots,\varepsilon_k$ in $\{\pm 1\}$, we have $\KK_w = \KK^{\varepsilon_1}_{x_1}\circ \KK^{\varepsilon_2}_{x_2}\circ\cdots \KK^{\varepsilon_k}_{x_k}$. For every $w$ in the free group on $X$, the linear map $\KK_w$ is in $GL(V)$ with determinant equal to $\pm 1$. Our first objective is to prove that

\begin{prop} The map $\KK: w\to \KK_w$ induces a linear representation of $G(X,S)$ in $GL(V)$ whose image is a finite group.
\end{prop}
\begin{proof} For every $x$ in $X$, the linear map~$\KK_x$ permutes the set $\{\pm x\mid x\in X\}$. So for every element~$w$ of $G(X,S)$, the linear map~$\KK_w$ permutes $\{\pm x\mid x\in X\}$ too. Since $X$ is a basis for $V$, the image of $\KK$ is finite.  Assume $S(x_1,x_2) = (x'_1,x'_2)$ with $x_1,x_2$ in $X$ and $x_1\neq x'_1$. As for Proposition~\ref{propHmap}, we need to prove that~$\KK_{x_1x_2} = \KK_{x'_1x'_2}$. Let $y$ be in $X$, and let us prove that~$\KK_{x_1x_2}(y) = \KK_{x'_1x'_2}(y)$. If $y\neq x_2$ and $y\neq x'_2$, then $\KK_{x_2}(y) = \HH_{x_2}(y)$  and $\KK_{x'_2}(y) = \HH_{x'_2}(y)$. We have  $S(\HH_{x_2}(y), x_2) = (\HH_{y}(x_2),y)$. If $\HH_{x_2}(y)$ was equal to $x_1$, we would have $S(\HH_{x_2}(y), x_2) = (x'_1,x'_2)$, a contradiction since $y\neq x'_2$. Thus, $\HH_{x_2}(y)\neq x_1$ and, therefore, $\KK_{x_1}(\KK_{x_2}(y)) = \KK_{x_1}(\HH_{x_2}(y)) = \HH_{x_1}(\HH_{x_2}(y))$. Similarly, $\KK_{x'_2}(y)\neq x'_1$ and $\KK_{x'_1}(\KK_{x'_2}(y)) = \KK_{x'_1}(\HH_{x'_2}(y)) =\HH_{x'_1}(\HH_{x'_2}(y))$. Therefore, $\KK_{x_1x_2}(y) = \KK_{x_1}(\KK_{x_2}(y)) = \HH_{x_1}(\HH_{x_2}(y)) = \HH_{x_1x_2}(y) = \HH_{x'_1x'_2}(y) =  \HH_{x'_1}(\HH_{x'_2}(y)) = \KK_{x_1}(\KK_{x_2}(y)) = \KK_{x'_1x'_2}(y)$. Now assume that $y = x_2$ or $y = x'_2$.  As $S$ is involutive, we may assume without restriction that $y = x_2$. Let $y_2$ be in $X$ so that $(y_2,x_2)$ is a frozen pair. We have $\KK_{x'_2}(x_2) = \HH_{x'_2}(x_2) = x'_1$. Therefore $\KK_{x'_1x'_2}(x_2) = \KK_{x'_1}(x'_1) = -\HH_{x'_1}(x'_1) = -\HH_{x'_1}(\HH_{x'_2}(x_2)) = -\HH_{x'_1x'_2}(x_2)$. On the other hand, $y_2\neq x_1$  since $(y_2,x_2)$ is a frozen pair whereas $(x_1,x_2)$ is not. Moreover, $\KK_{x_2}(x_2) = -\HH_{x_2}(x_2) = -y_2$.  Therefore, $\KK_{x_1x_2}(x_2) = -\KK_{x_1}(y_2) = -\HH_{x_1}(y_2) = -\HH_{x_1}(\HH_{x_2}(x_2)) = -\HH_{x_1x_2}(x_2)$. By Proposition~\ref{propHmap}, we deduce that~$\KK_{x_1x_2}(x_2) = \KK_{x'_1x'_2}(x_2)$.
\end{proof}
In the sequel, we still denote by~$\KK: G(X,S)\to GL(V),\ a\mapsto \KK_a$ the morphism induced by $\KK:w\mapsto \KK_w$. By $W(X,S)$ we denote the subgroup~$\KK(G(X,S))$ of $GL(V)$. So, for every $x$ in $X$ and every $a$ in $G(X,S)$, we have $\KK_a(x) = \pm\HH_a(x)$. For $\rho \in GL(V)$, we set $$n(\rho) = \# \{x\in X\mid \rho(x) \not\in X\}.$$ We turn now to the proof that $W(X,S)$ is a generating generated group for~$M(X,S)$. We denote by $\XX$ the set $\{\KK_x\mid x\in X\}$. Then, $\XX$ is a generating set for the group~$W(X,S)$. It follows from both facts that $W(X,S)$ is a finite group and $\XX$ is a generating set of the group~$W(X,S)$ that $\XX$ is also a generating set of~$W(X,S)$ considered as a monoid (every element of $\XX$ is of finite order, so its inverse is equal to some of its positive power). As a consequence, $$W(X,S) = \KK(M(X,S)).$$

\begin{lem}\label{lem_prepa}
(1) For every~$a,a'$ in~$M(X,S)$, one has~$n(\KK_{aa'}) \leq n(\KK_a)+n(\KK_{a'})$.\\
(2) For every~$a$ in~$M(X,S)$, one has~$\ell_\sX(a) \geq \ell_{\sXX}(\KK_a) \geq n(\KK_a)$.\\
(3) For every~$a$ in~$M(X,S)$ and every~$x$ in~$X$, if~$\KK_a(x) = -\HH_a(x)$ then~$x$ right divides~$a$ in~$M(X,S)$.\\
(4) For every~$a$ in~$M(X,S)$, if $\ell_{\sXX}(\KK_a)) = n(\KK_a)$, then $\KK_a = \KK_b$ for some $b$ in~$\Div{\Delta}$.
\end{lem}
\begin{proof}
(1) Let $a,a'$ lie in~$M(X,S)$, set $X_1 = \{x\in X\mid \KK_{a'}(x)\not\in X\}$ and $X_2 = \KK^{-1}_{a'}(\{x\in X\mid \KK_{a}(x)\not\in X\})$. We have $\# X_1 = n(\KK_{a'})$ and $\# X_2 = n(\KK_{a})$. It is easy to see that $\KK_{aa'}(x)\not\in X$ implies  $x\in X_1\cup X_2$. Therefore~$n(\KK_{aa'}) \leq n(\KK_a)+n(\KK_{a'})$.\\
(2) Let $a$ be in $M(X,S)$. If $a = x_1\cdots x_k$ with $x_1\ldots, x_k$ in $X$, then $\KK_a = \KK_{x_1}\circ\cdots\circ\KK_{x_k}$. Thus,~$\ell_\sX(a) \geq \ell_{\sXX}(\KK_a))$.  Now, $\ell_{\sXX}(\KK_a)) \geq n(\KK_a)$ by (1): if $\KK_a = \KK_{y_1}\circ\cdots\circ\KK_{y_r} = \KK_{y_1\cdots y_r}$ with $r = \ell_\sXX(\KK_a)$ and $y_1,\ldots, y_r$ in $X$, then $n(\KK_a)\leq \sum_{i = 1}^r n(\KK_{y_i}) = r$.\\
(3) Let $a$ lie in $M(X,S)$. Assume that $x$ belongs to~$X$ and is such that~$\KK_a(x) = -\HH_a(x)$. Assume $a = y_1\cdots y_k$ with $y_1,\cdots, y_k$ in $X$.  Set $a_{k+1} = 1$, and for $i = 1,\cdots, k$, set $a_i = y_i\cdots y_k$. Since~$\KK_{a_1}(x) = -\HH_{a_1}(x)$ and $\KK_{a_{k+1}}(x) = \HH_{a_{k+1}}(x)$, there exists $j\in \{1,\cdots, k\}$ such for $\KK_{a_{j+1}}(x) = \HH_{a_{j+1}}(x)$ and $\KK_{a_{j}}(x) = -\HH_{a_{j}}(x)$ that is,  $\KK_{y_j}(\HH_{a_{j+1}}(x)) = -\HH_{y_j}(\HH_{a_{j+1}}(x))$.  By definition of $\KK_{y_j}$, this means that $y_j = \HH_{a_{j+1}}(x)$. By lemma~\ref{lem_fct_H}, $x$ right-divides $\HH_{a_{j+1}}(x)a_{j+1}$, that is $y_ja_{j+1}$, in $M(X,S)$. Hence, it right divides $a_j$, and $a$ in $M(X,S)$.\\
(4) Assume $\ell_{\sXX}(\KK_a)) = n(\KK_a)$ for some~$a$ in~$M(X,S)$. Write $\KK_a = \KK_{y_1}\cdots \KK_{y_k}$ with $k = n(\KK_a)$ and $y_1,\cdots, y_k$ in $X$.  Set $b = y_1\cdots y_k$ in $M(X,S)$. By (3), there are $k$ distinct elements~$x_1,\cdots, x_k$ in $X$ that right divide $b$. Therefore the left lcm of~$x_1,\cdots, x_k$ right divides $b$. But this lcm is in~$\Div{\Delta}$ and its length is~$k$ by Proposition~\ref{prop_simple_lcm1}. Therefore it is equal to~$b$ and the latter belongs to~$\Div{\Delta}$.
\end{proof}

\begin{prop}\label{prop_clef} (1) Let $a$ lie in $\Div{\Delta}$. We have $\ell_\sX(a) = \ell_\sXX(\KK_a) = n(\KK_a)$. Moreover, for $x$ in $X$, $\KK_a(x) = -\HH_a(x)$ if $x$ right divides $a$, and  $\KK_a(x) = \HH_a(x)$ otherwise.\\
(2) The restriction of $\KK$ to $\Div{\Delta}$ is into.\\
(3) The element~$\KK_\Delta$ is $\XX$-balanced in $W(X,S)$ and $\lDiv{\sXX}{\KK_\Delta} = \{\KK_a\mid a\in \Div{\Delta}\}$.\end{prop}
\begin{proof} By Lemma~\ref{lem_prepa}, we have$$\{x\in X\mid \KK_a(x)\neq \HH_a(x)\} \subseteq \{x\in X\mid x\textrm{ right divides }a\}$$ and $\ell_\sX(a) \geq \ell_{\sXX}(\KK_a)) \geq n(\KK_a)$. Since $\ell_\sX(a) = \#\{x\in X\mid x$ right divides a$\}$ and $n(\KK_a) = \#\{x\in X\mid \KK_a(x)\neq \HH_a(x)\}$, in order to prove (1), we only need to prove that if $x$ right-divides $a$, then $\KK_a(x) \neq \HH_a(x)$. Assume $x$ right-divides $a$ in $M(X,S)$. Write $a = a_1x$ with $a_1$ in $M(X,S)$. We have $\KK_a(x)=\KK_{a_1}(\KK_x(x)) = -\KK_{a_1}(\HH_x(x))$. Since $a_1x$ belongs to $\Div{\Delta}$ and $(\HH_x(x),x)$ is a frozen pair, the element~$\HH_x(x)$ cannot right-divide $a_1$, from Proposition~\ref{prop_simple_lcm2}(2). By Lemma~\ref{lem_prepa}(3) this implies that $\KK_{a_1}(\HH_x(x)) = \HH_{a_1}(\HH_x(x))$. Thus,  $\KK_a(x) = -\KK_{a_1}(\HH_x(x)) = -\HH_{a_1}(\HH_x(x)) = -\HH_{a}(x)$, and (1) holds.\\
(2) is a direct consequence of (1) since every element of $\Div{\Delta}$ is the left lcm of its set of right-divisors that belong to $X$.\\
(3) Let $a$ be in $\Div{\Delta}$. Then there exist $a',a''$ in $\Div{\Delta}$ so that $aa' = a''a = \Delta$. Therefore, $\KK_a\KK_{a'} = \KK_{a''}\KK_a = \KK_\Delta$. Since $\ell_\sX(a)+\ell_\sX(a') = \ell_\sX(a'')+\ell_\sX(a) = \ell_\sX(\Delta)$, it follows from (1), that $\KK_a$ is both a $\sXX$-left factor and a $\sXX$-right factor of $\KK_\Delta$. Now, let us prove that the $\sXX$-left factors, and the $\sXX$-right factors of $\KK_\Delta$ belong to~$\KK(\Div{\Delta})$. Assume $\rho,\rho'$ belong to $W(X,S)$ so that $\KK_\Delta = \rho\rho'$ with $\ell_\sXX(\KK_\Delta) = \ell_\sXX(\rho)+\ell_\sXX(\rho')$. Since $W(X,S) = \KK(M(X,S))$, there exist $a$ and $a'$ in $M(X,S)$ so that $\rho = \KK_a$ and $\rho' = \KK_{a'}$. Then, it follows from~(1) and Lemma~\ref{lem_prepa}(1)(2) that $n(\KK_\Delta)\leq n(\KK_a)+n(\KK_{a'}) \leq \ell_\sXX(\KK_a)+\ell_\sXX(\KK_{a'}) = \ell_\sXX(\KK_\Delta) = n(\KK_\Delta)$. This imposes $n(\KK_{a}) = \ell_\sXX(\KK_a)$ and $n(\KK_{a'}) = \ell_\sXX(\KK_{a'})$. By Lemma~\ref{lem_prepa}(4) we deduce there exist $b,b'$ in $\Div{\Delta}$ so that $\KK_a = \KK_b$ and $\KK_{a'} = \KK_{b'}$. This proves~(3).
 \end{proof}

\begin{prop}\label{prop_quot_gene}
 Assume $M$ is a Garside monoid with Garside element~$\Delta$, with Garside group~$G$ and with atom set~$X$. Assume the length function~$\lX$ on $M$ is additive and $W$ is a quotient of $G$ that is a finite group. Set~$\XX = \psi(X)$, where $\psi: G\to W$ is the canonical morphism. Assume (a) the restriction of $\psi$ to $\Div{\Delta}$ is into; (b) for every $g$ in $\Div{\Delta}$, $\lX(g) = \ell_{\sXX}(\psi(g))$; (c) $\psi(\Delta)$ is $\XX$-balanced in $W$ and $\lDiv{\sXX}{\psi(\Delta)} = \psi( \Div{\Delta})$. Then $M_{\psi(\Delta),\XX}$ is a Garside monoid that is isomorphic to~$M$; more precisely, the map $x\mapsto \KK(x)$ induces an isomorphism from $M$ onto $M_{\psi(\Delta),\XX}$. In other words, $W$ is a generating generated group for~$M$.
\end{prop}
\begin{proof} We first remark that $\XX$ generates $W$ since $X$ generates $G$ and~$\XX = \psi(X)$. Set $w_0 = \psi(\Delta)$ and $M_0 = M_{w_0,\XX}$. By definition, $M_0$ has a monoid presentation $$\left\langle \underline{\xx}, \xx\in \lDiv{\sXX}{w_0} \mid \underline{\xx}\,\underline{\zz} = \underline{\yy}; \xx,\yy,\zz\in W \textrm{ with }\left\{\begin{array}{l} \xx\zz = \yy;\\ \ell_\sXX(\xx)+\ell_\sXX(\zz) = \ell_\sXX(\yy)\end{array}\right. \right\rangle.$$ But $\xx\mapsto \underline{\xx}$ is one-to-one, and $\psi$ is a morphim that induces a one-to-one map from $\Div{\Delta}$ to $\lDiv{\sXX}{w_0}$ such that $\ell_\sX(g) = \ell_\sXX(\phi(g))$ for $g$ in $\Div{\Delta}$. Therefore, the monoid~$M_0$ is isomorphic to the monoid defined by the monoid presentation $$\left\langle \underline{x}, x\in \Div{\Delta} \mid \underline{x}\, \underline{z} = \underline{y}; x,y,z\in \Div{\Delta} \textrm{ with }\left\{\begin{array}{l} xz = y;\\ \ell_\sX(x)+\ell_\sX(z) = \ell_\sX(y)\end{array}\right. \right\rangle.$$ The length is additive in $M$, so it turns out that the latter monoid has presentation $$\left\langle \underline{x}, x\in \Div{\Delta} \mid \underline{x}\, \underline{z} = \underline{y}; x,y,z\in \Div{\Delta} \textrm{ with } xz = y \right\rangle.$$ But this presentation is a presentation for $M$ by~\cite{ddkgm,bessis2}. Hence, $M_0$ is isomorphic to $M$, and thereby $(W,\XX)$ is a generating group for $M$.
\end{proof}
Gathering Propositions~\ref{prop_clef} and~\ref{prop_quot_gene} we get
\begin{cor}\label{cor_basedon} The monoid~$W(X,S)$ is a generating generated group for $M(X,S)$.
\end{cor}
\section{The case of generating generated sections}
\label{sec_ftquot}
In the previous section, we proved that every Garside monoid~$M(X,S)$ of $I$-type admits a finite generating generated group~$W(X,S)$.  In the case of spherical type Artin-Tits monoids, the associated generating generated groups are the Coxeter groups. Moreover, the latter are also generating generated sections. In this section, we investigate the properties of the groups $W(X,S)$. In particular, we are interested to know which properties they share with Coxeter groups. We show that the groups $W(X,S)$ are not necessarily generating generated sections of~$M(X,S)$. Yet, we can characterize whenever it occurs, by a condition that can be tested on the presentation of~$M(X,S)$.
Here are the questions we focus on in the remaining of the section:
 \begin{ques} \label{3questions} Let $(X,S)$ be a non-degenerate symmetric set-theoretical solution of the QYBE. Denote by $\Delta$ the Garside element of $M(X,S)$, and by $\KK: M(X,S)\to W(X,S), a\mapsto\KK_a$ the surjective morphism defined in the previous section.
\begin{enumerate}
\item[(a)] Is there a simple necessary and sufficient condition that ensures that $W(X,S)$ is a generating generated section for~$M(X,S)$ ?
\item[(b)] Considering the exact sequence \begin{equation}1\to N(X,S)\to G(X,S)\stackrel{\psi}{\to} W(X,S)\to 1\label{suiteexacte}\end{equation} What is the structure of the group $N(X,S)$, that is $\textrm{Ker}(\psi)$?
\item[(c)] What is the cardinality of $W(X,S)$ ?
\item[(d)] Can we find a presentation of $W(X,S)$ with $X$ as a generating set?
\end{enumerate}
\end{ques}

As in the previous section, we fix a non-degenerate symmetric set-theoretical solution~$(X,S)$ of the QYBE, where $X$ is a finite set of cardinality $n$. We still denote by $\Delta$ the Garside element of $M(X,S)$, and by $\KK:a\mapsto \KK_a$ the surjective morphism from $M(X,S)$ to $W(X,S)$.
\subsection{A necessary and sufficient condition for $W(X,S)$ to be a generating generated section for $M(X,S)$}
Here we introduce a property, namely Property~$\CC$, and prove that $M(X,S)$ satisfies this property if and only if $W(X,S)$ is a generating generated section for $M(X,S)$. Recall that the maps~$f_x$ and~$g_x$ have been defined in Section~\ref{secQYBE}.
\begin{defn} (1) We say Property~$\CC$ holds for a pair~$(x,y)$ of elements in~$X$ if $g_x\circ g_y= Id_X$ and $f_y\circ f_x= Id_X$.\\
(2) We say that~$(X,S)$ verifies Property~$\CC$ if say Property~$\CC$ holds for each frozen pair.
\end{defn}
In Example \ref{exemple:exesolu_et_gars}, the solution $(X,S)$ verifies Property~$\CC$. Indeed, the frozen words are $x_{1}^{2}$, $x^2_{2}$, $x_{3}x_{4}$, $x_{4}x_{3}$ and Property~$\CC$ holds for each of them.
\begin{rem}The Property~$\CC$ is not verified by all non-degenerate symmetric set-theoretical solutions as shown by Example~\ref{ex:pasCC}. However, this is a property that is satisfied by various solutions (see~\cite[Ex.~1.12]{gateva}, \cite[Ex.~2.3]{chou_godel} or the example after Prop.~4.2 of~\cite{jespers} for instance.)
\end{rem}

\begin{prop} \label{lem_cns} The monoid~$W(X,S)$ is a generating generated section for~$M(X,S)$ if and only if Property~$\CC$ is verified. Moreover, in this case, for every frozen pair~$(x,y)$, the element~$xy$ belongs to $\textrm{Ker}(\KK)$. 
\end{prop}
\begin{proof}
Assume~$W(X,S)$ is a generating generated section for~$M(X,S)$. Then $\KK(\Div{\Delta}) = W(X,S)$. Let~$xy$ be any frozen word. By assumption, there exists $a$ in $\Div{\Delta}$ so that $\KK_{xy} = \KK_a$. Let $z$ belong to~$X$. If $z = y$, then  $\KK_{xy}(y) = \KK_{x}(-x) = \phi_x(x)\in X$. Otherwise, $\KK_{xy}(z) = \KK_{x}(\HH_y(z))$. Since $z\neq y$ and $\HH_y(y) = x$, we have $\HH_y(z))\neq x$. Therefore, $\KK_{x}(\HH_y(z)) = \HH_x(\HH_y(z)) = \HH_{xy}(z)\in X$. Hence, $\KK_{xy}(X) \subseteq X$. We conclude that $\KK_{xy}(X) =  X$ and, thereby, that $n(\psi_{xy})=0$. By Proposition~\ref{prop_clef}(1), this imposes $a = 1$ and~$xy$ belongs to $\textrm{Ker}(\KK)$. It follows that $\KK_x = \KK_y^{-1}$ and $f^{-1}_x\circ f^{-1}_y= Id_X$, by definition of morphisms~$\KK$ and $\HH$. Hence, $f_y\circ f_x= Id_X$.  Now let $z$ belong to $X$. Then, we have $xyz = xz'y'$  where $y' = f_z(y)$ and $z' = g_y(z)$. We have also $xz'y' = tx'y'$ with $x' = g_{z'}(x)$ and $t = g_x(z')$. In particular we have $xyz\,=\,tx'y'$, $t = g_x\circ g_y(z)$ and  $z = f_{y'}\circ f_{x'}(t)$. Now, we claim that $x'y'$ is a frozen word. We consider two cases, depending whether $yz$ is a frozen word or not. Assume, first,  $yz$ is a frozen word. In this case, the element $xyz$ as a unique word representative, that is $xyz$. Hence, there is nothing to prove: $x=t$, $y = x'$ and $z = y'$. So $x'y'$ is a frozen word. Assume, secondly, that $yz$ is not a frozen word. This imposes $y'\neq z$ and $z'\neq y$. Since $xy$ is a frozen word, it follows that $xz'$ is not. Since $x' = g_{z'}(x)$, we get  $x'\neq z'$. Set $x'' = g_{x'}(y')$ and  $y'' = f_{y'}(x')$. We have $x'y' = x''y''$. Since $xy$ is a frozen word the element~$xyz$ does not belong to~$\Div{\Delta}$ by Proposition~\ref{prop_simple_lcm1}. Since~$xyz$ is not in~$\Div{\Delta}$ and the length of~$xyz$ is $3$, by Proposition~\ref{prop_simple_lcm1}, the set of its right-divisors that belongs to~$X$ is of cardinality at most~$2$. Since~$xyz = xz'y'=tx'y'=tx''y''$ and $y'\neq z$, we must have either $y'' = y'$ or $y'' = z$. But $y'' = f_{y'}(x')$ and $z = f_{y'}(z')$. Since $x'\neq z'$, it follows that $z\neq Z''$. Thus, $y' = y''$. Since $y'' = f_{y'}(x')$, this means that $x'y'$ is a frozen word. This proves our claim. As a consequence $f_{y'}\circ f_{x'} = Id_X$, $z = t$ and $g_x\circ g_y(z) = z$. Hence, $g_x\circ g_y = Id_X$.

Conversely, assume Property~$\CC$ is verified by $(X,S)$. Let $\rho$ be in $W(X,S)$. There is an element~$a$ in $M(X,S)$ so that $\rho = \KK_a$ (see remark before Lemma~\ref{lem_prepa}). Assume we can write $a = a_1xya_2$ with $x,y$ in $X$, $a_1,a_2$ in~$M(X,S)$ so that $xy$ is a frozen word. Since $f_y\circ f_x = Id_X$, we have $\KK_{xy} = Id_X$ and, therefore, $\rho =\KK_a =  \KK_{a_1a_2}$ with $\ell_\sX(a_1a_2)<\ell_\sX(a)$. Hence, if $a$ is of minimal length among the elements of $M(X,S)$ whose image by $\KK$ is $\rho$, then none of its representative words  contains a frozen word as a subword. By Proposition~\ref{prop_simple_lcm2}, this means that $a$ belongs to~$\Div{\Delta}$. Thus $\rho$ belongs to $\KK(\Div{\Delta})$, and $W(X,S) = \KK(M(X,S)) = \KK(\Div{\Delta})$. Hence,~$W(X,S)$ is a generating generated section for~$M(X,S)$.
\end{proof}
\subsection{The frozen subgroup of~$M(X,S)$}
Our purpose is to prove that, when Property~$\CC$ is verified, the subgroup of $G(X,S)$ generated by the frozen elements is a normal subgroup and a free Abelian group, freely generated by the frozen elements. This subgroup will turn out to be, under Property~$\CC$, the subgroup~$N(X,S)$ in the exact sequence~(\ref{suiteexacte}). Before proceeding, we need to prove some properties satisfied by frozen words.

\begin{lem}\label{lem_x_ix_j_trivial}
Assume $(X,S)$ verifies Property~$\CC$ and $xy$ is a frozen word then\\
(1) the word $yx$ is frozen.\\
(2) Let $z$ be in $X$. There exists a unique pair $(x',y')$ so that $xyz = zx'y'$ and~$x'y'$ is a frozen word.
\end{lem}
\begin{proof} (1) Assume $xy$ is a frozen word. Applying Property~$\CC$, we get $g_x\circ g_y(x) = x$, that is $g_x(g_y(x)) = x$. Therefore, $(x,g_y(x))$ is a frozen pair. Since $(x,y)$ and $(x, g_y(x))$ are frozen pairs, it follows that $y = g_y(x)$. This imposes that $yx$ is a frozen word.\\
(2) The unicity is clear by the cancellativity property. If $z = x$, then $yx$ is a frozen word by~(1). So, assume~$z\neq x$. Consider $x',y',z'$ and $t$ like in the proof of Proposition~\ref{lem_cns}. We have $xyz = tx'y'$, where $x'y'$ is a frozen word and $t = g_x\circ g_y(z)$. Since Property~$\CC$ holds for the frozen pair $(x,y)$, one has~$g_x\circ g_y=Id$, so $z=t$ and $xyz = zx'y'$, with~$x'y'$ frozen.
\end{proof}
In the sequel, it will be convenient to introduce a notation for the frozen words. So we denote by $\fr{1},\cdots, \fr{n}$ the $n$ distinct frozen words. We denote in the same way the associated frozen elements of $M(X,S)$.
\begin{lem}\label{lem_importance_conditions_fg_2}
Assume $(X,S)$ verifies Property~$\CC$.\\(1) For any $i,j$ distinct in~$\{1,\ldots,n\}$, the elements $\fr{i}\fr{j}$ and  $\fr{j}\fr{i}$ are equal in $M(X,S)$ and are both the right lcm and the left lcm of $\fr{i}$ and $\fr{j}$.\\(2) Let~$a = \fr{1}^{m_1}\cdots \fr{n}^{m_n}$ be in $M(X,S)$ such that $m_1,\cdots, m_n$ are not negative integers. Assume that  $\fr{i}$ left divides $a$, where $i \in~\{1,\cdots, n\}$. Then, $m_i\geq 1$.
\end{lem}
\begin{proof}Let $(x_1,y_1)$ and $(x_2,y_2)$ be the distinct frozen pairs such that $\fr{i} = x_1y_1$ and $\fr{j} = x_2y_2$. Note that, by Lemma~\ref{lem_x_ix_j_trivial}, $(y_1,x_1)$ and $(y_2,x_2)$ are also frozen pairs. Let $z_1,z_2,t_1,t_2$ be such that $x_1z_1 = x_2z_2$ and $t_1y_1 = t_2y_2$ are defining relations. Let us first prove that $y_1x_2 = z_1t_2$ and $y_2x_1 = z_2t_1$. We  have  
$x_1z_1 = x_2z_2$ and, therefore, $x_2 = g_{x_1}(z_1)$ and $x_1= g_{x_2}(z_2)$. As a consequence, we have $g_{y_1}(x_2) = g_{y_1}\circ g_{x_1}(z_1) = z_1$ and $g_{y_2}(x_1) =  g_{y_2}\circ g_{x_2}(z_2) = z_2$.
Similarly, we get also $f_{x_2}(y_1) = f_{x_2}\circ f_{y_2}(t_2) = t_2$ and $f_{x_1}(y_2) = f_{x_1}\circ f_{y_1}(t_1) = t_1$. Gathering the equalities $g_{y_1}(x_2) = z_1$ and $f_{x_2}(y_1) = t_2$, we get the expected equalities.
So, we conlude that $\fr{i}\fr{j} = x_1y_1x_2y_2 = x_1z_1t_2y_2 = x_2z_2t_1y_1 = x_2y_2x_1y_1=\fr{j}\fr{i}$. Now consider the right lcm of $\fr{i}$ and $\fr{j}$. It has to left divide $\fr{i}\fr{j}$, that is $x_1y_1x_2y_2$, and to be a right multiple of $\fr{i}$, which is $x_1y_1$. An enumeration of all the representative words of $x_1y_1x_2y_2$ proves that the only ones that have $x_1y_1$ as a prefix are $x_1y_1x_2$ and $x_1y_1x_2y_2$.  As the unique representative word of the frozen element~$x_1y_1$ is the frozen word~$x_1y_1$, it follows that the right lcm of $x_1y_1$ and $x_2y_2$ is either $x_1y_1x_2$ or $x_1y_1x_2y_2$. Now, the element $x_1y_1x_2$ has only three representative words and the frozen word~$x_2y_2$ is a prefix of none of them. Hence, $x_1y_1x_2$ is not a right mutiple of the frozen element~$x_2y_2$. Thus, the right lcm of $x_1y_1$ and $x_2y_2$ is $x_1y_1x_2y_2$. By symmetry, it is also their left lcm.\\
(2) We prove the results by induction on the sum~$m = \sum_{i = 1}^nm_i$. As $M(X,S)$ has no invertible elements, except the identity, the case $m = 0$ is not possible. The case $m = 1$ is trivial since a frozen element has a unique word representative, its associated frozen word. So, assume $m\geq 2$. As the frozen elements commute, we can, up to a permutation of the indices, assume that $m_1>0$. If $i = 1$, the result trivially holds. Assume $i\neq 1$. Then the right lcm of $\fr{1}$ and $\fr{i}$ has to left divide $a$. But this lcm is $\fr{1}\fr{i}$, from $(1)$. By the cancellativity property, we get that $\fr{i}$ left divides~$\fr{1}^{m_1-1}\cdots \fr{n}^{m_n}$, which imposes, by the induction hypothesis, that  $m_i\geq 1$.
\end{proof}

\begin{prop}\label{prop_submonoid_abelian_mod}
Assume $(X,S)$ verifies Property~$\CC$.\\
(1) The action of $G(X,S)$ on itself by conjugation permutes the frozen elements.\\
(2) Let $N^{\scriptscriptstyle +}(X,S)$ be the submonoid of $M(X,S)$ generated by the frozen elements. Then $N^{\scriptscriptstyle +}(X,S)$ is a free Abelian monoid generated by the frozen elements.\\
(3) The right lcm and the left lcm in $M(X,S)$ of any two elements of~$N^{\scriptscriptstyle +}(X,S)$ are equal and belong to~$N^{\scriptscriptstyle +}(X,S)$.
\end{prop}
\begin{proof} (1) It follows from Lemma~\ref{lem_x_ix_j_trivial}(2) that for every element~$x$  of~$X$ and every frozen element~$\fr{i}$, the element $x\fr{i}x^{-1}$ is a frozen element. Since $X$ generates~the group~$G(X,S)$, the conjugation action of~$G(X,S)$ permutes the frozen elements.\\ (2) By Lemma~\ref{lem_importance_conditions_fg_2}(1) the frozen elements commute. Now, consider an equality~$\fr{1}^{p_1}\cdots \fr{n}^{p_n} = \fr{1}^{q_1}\cdots \fr{n}^{q_n}$ in $M(X,S)$ where $p_i$ and $q_i$ are non negative integers for $i = 1,\ldots, n$. Let us prove that $p_1 = q_1, \ldots, p_n = q_n$. Using the cancellativity property and the commutativity of the frozen elements, we can assume without restriction that $\min(p_1,q_1) = 0, \ldots, \min(p_n,q_n) = 0$. But in this case, Lemma~\ref{lem_importance_conditions_fg_2} imposes that we must have $p_i = q_i = 0$ for every $i$ in~$ 1,\ldots, n$. Hence, the relation is trivial and $N^{\scriptscriptstyle +}(X,S)$ is freely generated as a monoid by the frozen elements.\\
(3) is derived from Lemma \ref{lem_importance_conditions_fg_2}(1) and is closed to the proof of Lemma \ref{lem_importance_conditions_fg_2}(2). One can prove by induction that the left lcm of~$\fr{1}^{p_1}\cdots \fr{n}^{p_n}$ and $\fr{1}^{q_1}\cdots \fr{n}^{q_n}$ is~$\fr{1}^{\max(p_1,q_1)}\cdots \fr{n}^{\max(p_n,q_n)}$. We left the details for the reader.
\end{proof}

\begin{cor}\label{cor_submonoid_abelian}
Assume $(X,S)$ verifies Property~$\CC$.\\
The subgroup $N$ of $G(X,S)$ generated by the frozen elements is a free Abelian group, freely generated by the frozen elements, and a normal subgroup of $G(X,S)$.
\end{cor}

\begin{proof}
Since $N(X,S)$ is generated by the frozen elements, Proposition~\ref{prop_submonoid_abelian_mod}(1) induces that $N$ is normal in~$G(X,S)$. As the frozen elements commute, the group $N$ is commutative. It remains to prove that it is freely generated by the frozen elements. Consider an equality~$\fr{1}^{m_1}\cdots \fr{n}^{m_n} = 1$ in $N$ with the $m_i$ in $\mathbb{Z}$. Then we get an equality~$\fr{1}^{p_1}\cdots \fr{n}^{p_n} = \fr{1}^{q_1}\cdots \fr{n}^{q_n}$ in $M(X,S)$ where $p_i$ and $q_i$ are non negative integers such that $\min(p_i,q_i) = 0$ and $\max(p_i,q_i) = |m_i|$  for $i = 1,\ldots, n$. It follows from Proposition~\ref{lem_importance_conditions_fg_2}(2) that me must have $m_i = p_i = q_i = 0$ for every $i$ in~$ 1,\ldots, n$. Hence, the relation is trivial and $N$ is freely generated by the frozen elements as an Abelian group.
\end{proof}

\subsection{A presentation for~$W(X,S)$ when Property~$\CC$ is verified}
\label{sec:present}
We are now ready to provide a group presentation for the group~$W(X,S)$, assuming Property~$\CC$ is verified.
Let us recall some notations of the previous section. By~$N$ we denote the frozen subgroup of $G(X,S)$ generated by the frozen elements. By~$N^{\scriptscriptstyle +}(X,S)$, we denote the frozen submonoid of $M(X,S)$ generated by the frozen elements. We also recall the exact sequence $$1\to N(X,S)\to G(X,S)\to W(X,S)\to 1.$$
Our first objective is to prove
\begin{prop}\label{label:propfrozensbgp} Assume~$(X,S)$ verifies Property~$\CC$. Then the normal subgroup~$N(X,S)$ is equal to the frozen subgroup~$N$ generated by the frozen elements.
\end{prop}

Before we proceed, we need the following result.
\begin{lem} \label{lem_prepeproofpropfrozensbgp}Assume~$(X,S)$ verifies Property~$\CC$. Every element~$a$ of~$M(X,S)$ can be decomposed as a product~$a = a_1a_2$ where~$a_1$ lies in~$N^{\scriptscriptstyle +}(X,S)$ and~$a_2$ belongs to~$\Div{\Delta}$.
\end{lem}
\begin{proof} We prove the result by induction on~$\ell_\sX(a)$. If $a = 1$, there is nothing to prove. Assume~$\ell_\sX(a) \geq 1.$ If $a$ belongs to~$\Div{\Delta}$, we can take $a_1 = 1$ and $a_2 = a$. If $a$ is not in $\Div{\Delta}$, then, by Proposition~\ref{prop_simple_lcm2}(2), we can write $a = a'xya''$ with $a',a''$ in $M(X,S)$ and $xy$ a frozen word. Then, by Proposition~\ref{prop_submonoid_abelian_mod}, there exists a frozen word $x'y'$ so that $a = x'y'a'a''$. Moreover, $\ell_\sX(a'a'')< \ell_\sX(a)$. By the induction hypothesis we can write $a'a'' = a'_1a_2$ with~$a'_1$ is in~$N^{\scriptscriptstyle +}(X,S)$ and~$a_2$ in~$\Div{\Delta}$. Then we have $a = a_1a_2$ where $a_2$ lies in~$\Div{\Delta}$ and $a_1$ is equal to $x'y'a'_1$ and, therefore, belongs to~$N^{\scriptscriptstyle +}(X,S)$.
\end{proof}

\begin{proof}[Proof of Proposition~\ref{label:propfrozensbgp}] We first remark that $N$ is a subgroup of $N(X,S)$, because $N$ is generated by the frozen elements, that belong to $N(X,S)$ by Proposition~\ref{lem_cns}. Conversely, let $a$ be in $G(X,S)$ that belongs to the kernel of~$\KK$. As $G(X,S)$ is the group of fraction of $M(X,S)$, there exists $b,c$ in $M(X,S)$ so that $a = bc^{-1}$. By Lemma~\ref{lem_prepeproofpropfrozensbgp}, we can write $b = b_1b_2$ and $c = c_1c_2$ with $b_1,c_1$ in~$N^{\scriptscriptstyle +}(X,S)$ and~$b_2,c_2$ in~$\Div{\Delta}$. Then we have $\KK_a = \KK_b\KK_{c^{-1}} = \KK_{b_1}\KK_{b_2}\KK_{c_2^{-1}}\KK_{c_1^{-1}} = \KK_{b_2}\KK_{c_2^{-1}}$. It follows that  $\KK_{b_2} = \KK_{c_2}$. By Proposition~\ref{prop_clef}(2), we get $b_2 = c_2$. Thus, $a = b_1c_1^{-1}$ and $a$ belongs to $N$. Hence, $N(X,S)\subseteq N$ and, finally, $N(X,S) = N$.  \end{proof}
We recall that by~$\XX$ we denote the set~$\{\KK_x\mid x\in X\}$.
\begin{cor} \label{main_th} If $(X,S)$ verifies Property~$\CC$, then\\
(1) The group~$W(X,S)$ has the group presentation
\begin{equation}\left\langle\ \XX\ \left|\ \begin{array}{lcl}\KK_x\KK_y = \KK_z\KK_t&;&x,y\in X, S(x,y) = (z,t)\neq (x,y)\\\KK_x\KK_y=1&;& x,y\in X, S(x,y) = (x,y) \end{array} \right.\right\rangle\label{equation:structuregroup2}\end{equation}
(2)  $W(X,S)$ is a finite group of order $2^n$, where $n$ is the cardinality of~$X$.\\
 \end{cor}

\begin{proof}
Point~(1) is a direct consequence of Proposition~\ref{label:propfrozensbgp} and of the defining presentation of $G(X,S)$, given in Definition~\ref{def_struct_gp}. The cardinality of $\Div{\Delta}$ is equal to $2^n$ by Proposition~\ref{prop_nombre_simple}. Therefore the cardinality of $W(X,S)$ is~$2^n$ by Proposition~\ref{lem_cns}.
\end{proof}

\begin{ex} Assume $X = \{x_1,x_2\}$ so that $S(x_i,x_j) = (x_j,x_i)$ for $i,j = 1,2$. Then $M(X,S)$ has the presentation~$\langle x_1,x_2\mid x_1x_2 = x_2x_1\rangle$ and $W(X,S) = \langle \xx_1,\xx_2\mid \xx_1^2 = \xx_2^2 = 1 ; \xx_1\xx_2 = \xx_2\xx_1\rangle$ where $\xx_1 = \KK_{x_1}$ and $\xx_2 = \KK_{x_2}$.
\end{ex}

When Property~$\CC$ is not verified, then the cardinality of $W(X,S)$ is not necessarily equal to~$2^n$ as it is shown by the following example.
\begin{ex} \label{ex:pasCC}
Let $X = \{1,2,3,4\}$ and consider the group define by the following presentation
$$\left\langle x_1,x_2,x_3,x_4\left| \begin{array}{ccc}x_1x_2=x_3x_1&;&x_2x_2 = x_4x_3\\x_1x_3=x_4x_1&;&x_3x_3 = x_2x_4\\x_1x_4=x_2x_1&;&x_4x_4 = x_3x_2\end{array}\right.\right\rangle$$
It is easy to check that this group is a group of $I$-type since this is the envelopping group of a monoid of $I$-type (Definition~\ref{def:Itype}). We have~$x_1^2x_2 = x_4x_1^2$ and~$x_1^2$ is a frozen word. Hense, Property~$\CC$ is not verified. Now, considered as a subgroup of the group of permutations of~$\{\pm1,\pm2,\pm3,\pm4\}$, the maps $\psi_1,\psi_2,\psi_3$ and $\psi_4$ are equal to $(1,-1)(2,3,4)(-2,-3,-4)$, $(2,-4,-3,-2,4,3)$, $(2,4,3,-2,-4,-3)$ and $(2,4,-3,-2,-4,3)$ respectively. One can check (using GAP for instance) that the cardinality of $W(X,S)$ is $48$, that is $2^4\times 3^2$. One can also check that the centre of $W(X,S)$ has four elements and is generated by the two elements~$(1,-1)$ and~$(2,-2)(3,-3)(4,-4)$.
\end{ex}

We conclude with some extra properties of the group~$W(X,S)$, showing this group shares several properties with Coxeter groups. For $\xx,\yy$ in $\XX$, we shall say that $(\xx,\yy)$ is a frozen pair when $\yy = \xx^{-1}$, that is when $(x,y)$ is a frozen pair. By $\SSS: \XX\times\XX\to\XX\times\XX$, we denote the map induced by the map~$S:X\times X\to X\times X$.  
\begin{cor} Assume~$(X,S)$ verifies Property~$\CC$.
(1) The group~$W(X,S)$ contains a unique element~$w_0$ of maximal length on~$\XX$. This element is $\phi(\Delta)$. Its length on~$\XX$ is $n$.\\
(2) The order of $w_0$ is two, and the conjugation action of $w_0$ permutes the elements of $\XX$.\\
\end{cor}
\begin{proof} Let $w$ belong to~$W(X,S)$. By Proposition~\ref{lem_cns}, there exists~$g$ in $\Div{\Delta}$ so that $\psi_g = w$. The length of $g$ on $X$ is at most~$n$ by Proposition~\ref{prop_simple_lcm1}, and the length of $g$ is equal to~$n$ if and only if $g = \Delta$. Thus, by Proposition~\ref{prop_clef}(1), the length on~$\XX$ of any element of $W(X,S)$ is at most $n$, and $\psi(\Delta)$ is the unique element of length~$n$ on $\XX$. For every $x$ in $X$, there exists (a unique) $y$ in $X$ so that $(x,y)$ is a frozen pair. By Property~$\CC$ we have $\KK_x\circ\KK_y = \KK_{xy} = 1$. Therefore, we have $\XX = \XX^{-1}$ in $W(X,S)$. This imposes that any element of $W(X,S)$ has the same length on $X$ as its inverse. But $w_0$ is the unique element of length~$n$. Therefore $w_0 = w_0^{-1}$, that is $w_0^2 = 1$. Moreover, Since $\Delta$ is a (actually \emph{the}) Garside element, we have $\Delta X = X\Delta$. Thereby, $w_0\XX w_0^{-1} = \XX$.
\end{proof}

\begin{prop}Assume~$(X,S)$ verifies Property~$\CC$.
 (1) For every $w$ in $W(X,S)$ and every $x$ in $\XX$, one has $\ell_\sXX(xw) = \ell_{\sXX}(w)\pm 1$. Moreover, if $\ell_\sXX(xw) = \ell_{\sXX}(w) - 1$, then there exists $w_1$ in $W(X,S)$ so that $w = \xx^{-1}w_1$ and $\ell_\sXX(w_1) = \ell_\sXX(w)-1$.\\
(2) Assume $w$ lies in $W(X,S)$ and $\xx,\yy$ lie in $\XX$ so that $\ell_\sXX(\xx w) = \ell_\sXX(w\yy) = \ell_\sXX(w)+1$ and $\ell_\sXX(\xx w\yy) = \ell_\sXX(w)$. Assume $w = \KK_{z_1}\cdots \KK_{z_k}$ with $z_1,\ldots,z_k$ in $X$ and  $k = \ell_\sXX(w)$. Then, there exist $y_1,\ldots,y_{k+1}$ in $X$ and $x_0,x_1,\ldots,x_k$ in $X$ so that $\KK_{x_0} = \xx$, $\KK_{y_{k+1}} = \yy$, $\KK_{z_i}\KK_{y_{i+1}} = \KK_{y_i}\KK_{x_i}$,  no $(z_i,y_{i+1})$ is a frozen pair  and $\psi_{y_1} = \xx^{-1}$. In particular, $\xx w\yy = \KK_{x_1}\cdots\KK_{x_k}$.

\end{prop}
\begin{proof} (1) and (2) are  direct consequences of Proposition~\ref{prop_simple_lcm2}.
\end{proof}
\begin{rem}
The group $W(X,S)$ is  a 2-group with order equal to~$2^n$, where $n$ is the cardinality of $X$. So, $W(X,S)$ is nilpotent and it has nilpotency class at most $n-1$. It is never cyclic,  because it is Abelian if and only if $W(X,S)$ is a trivial solution, and in this later case $W(X,S)$ is isomorphic to~$(\mathbb{Z}_2)^n$.
\end{rem}

\begin{ex}(1) Consider Example \ref{exemple:exesolu_et_gars}. The exponent of $W(X,S)$ is $2^3$ and its nilpotency class is~$3$. \\
(2) Consider the trivial solution $(X,S)$ with $X$ of cardinality $n$, the structure group of $(X,S)$ is the free Abelian group on $n$ generators and the finite quotient group $W(X,S)$ is $(\mathbb{Z}_2)^n$. Its nilpotency class is $1$ and its exponent is $2$.\\
(3) consider the following \emph{almost trivial} solution $(X,S)$ with $X = \{1,\ldots,6\}$ and $g_i=f_i=Id_X$ for $1 \leq i \leq 4$ and $f_5=f_6=g_5=g_6=(5,6)$. The structure group is isomorphic to~$(\mathbb{Z})^4 \times\langle x_5,x_6 \mid x_5^2=x_6^2\rangle$ and the finite quotient group $W(X,S)$ is $(\mathbb{Z}_2)^4\times\mathbb{Z}_4$. Its nilpotency class is $1$ and its exponent is $4$.\end{ex}
  

\bigskip\bigskip\noindent
{ Fabienne Chouraqui,}

\smallskip\noindent
Technion, Haifa, Israel.
\smallskip\noindent
E-mail: {\tt fabienne@tx.technion.ac.il}

\bigskip\bigskip\noindent
{ Eddy  Godelle,}

\smallskip\noindent
Universit\'e de Caen, UMR 6139 du CNRS, LMNO, Campus II, 14032 Caen cedex, France.
\smallskip\noindent
E-mail: {\tt eddy.godelle@unicaen.fr}

\begin{thebibliography}{40}
\bibitem{bessis2} D.~Bessis, The Dual Braid monoid. Annales de l'ENS {\bf 36} (200) 647--683.
\bibitem{chou_art} F.~Chouraqui, Garside groups and the Yang-Baxter equation, Comm. in Algebra {\bf 38} (2010) 4441-4460.
\bibitem{chou_godel} F.~Chouraqui and E.~Godelle, Folding of set theoretical solutions of the Yang-Baxter Equation, Algebra and Representation Theory {\bf 15} (2012) 1277--1290.
\bibitem{Clifford} A.H.~Clifford, G.B.~Preston, The algebraic theory of semigroups, vol.1, Mathematical Surveys {\bf 7}, AMS, Providence R.I., 1961.
\bibitem{deh_francais} P.~Dehornoy, Groupes de Garside, Ann. Scient. Ec. Norm. Sup. {\bf 35} (2002) 267-306.
\bibitem{ddkgm} P.~Dehornoy, F.~Digne, E.~Godelle and J.~Michel, Garside Theory, http://www.math.unicaen/$\sim$garside.
\bibitem{DePa}  P.~Dehornoy and L.~Paris, Gaussian groups and Garside groups, two generalisations of Artin groups, Proc. London Math Soc.  {\bf 79} (1999) 569--604.
\bibitem{DiM} F.~Digne and J.~Michel, Garside and Locally Garside categories, preprint december 2006. Arxiv.math.GR/0612652.
\bibitem{etingof} P.~Etingof, T.~Schedler, A.~Soloviev, Set-theoretical solutions to the Quantum Yang-Baxter equation, Duke Math. J. {\bf 100} (1999) 169-209.
\bibitem{garside} F.A.~Garside, The braid group and other groups, Quart.~J.~Math.~Oxford { \bf 20} (1969) 235-254.
\bibitem{gateva} T.~Gateva-Ivanova, Garside structures on monoids with quadratic square free relations, Algebra and Representation theory {\bf 14} (2011) 779-802.
\bibitem{gateva98} T.~Gateva-Ivanova and M.~Van~den~Bergh, Semigroups of $I$-type, J.~Algebra {\bf 206} (1998) 97-112.
\bibitem{jespers} E.~Jespers, J.~Okninski, Monoids and groups of $I$-type, Algebra and Representation Theory {\bf 8} (2005) 709-729.
\bibitem{michel} J.~Michel, cours de DEA 2004 Paris VII, http//www.math.jussieu.fr/$\sim$jmichel.
\bibitem{tits} J.~Tits, Normalisateur de tores. I: groupe de Coxeter \'etendus, J.~Algebra {\bf 4} (1966) 96-116.
  \end{thebibliography}
\end{document}